\documentclass[12pt, reqno]{amsart}

\usepackage{amsmath, amsthm, amscd, amsfonts, amssymb, graphicx, color}
\usepackage[bookmarksnumbered, colorlinks, plainpages]{hyperref}
\input{mathrsfs.sty}
\hypersetup{colorlinks=true,linkcolor=red, anchorcolor=green, citecolor=cyan, urlcolor=red, filecolor=magenta, pdftoolbar=true}

\newtheorem{theorem}{Theorem}[section]
\newtheorem{lemma}[theorem]{Lemma}

\newtheorem{corollary}[theorem]{Corollary}
\theoremstyle{definition}

\theoremstyle{remark}
\newtheorem{remark}[theorem]{Remark}
\numberwithin{equation}{section}

\newcommand{\uinorm}[1]{{\left\vert\kern-0.25ex\left\vert\kern-0.25ex\left\vert #1
 \right\vert\kern-0.25ex\right\vert\kern-0.25ex\right\vert}}

\begin{document}

\title[Ando--Hiai inequality]{Power means of probability measures and Ando--Hiai inequality}
\author[ M. Kian]{ Mohsen Kian$^1$  \MakeLowercase{and} Mohammad Sal Moslehian}
\address{$^1$ Department of Mathematics, Faculty of Basic Sciences, University of Bojnord, P. O. Box
1339, Bojnord 94531, Iran}
\email{kian@ub.ac.ir }
\address{$^2$ Department of Pure Mathematics, Center of Excellence in Analysis on Algebraic
Structures (CEAAS), Ferdowsi University of Mashhad, P. O. Box 1159, Mashhad 91775,
Iran.}
\email{moslehian@um.ac.ir}


\begin{abstract}
Let $\mu$ be a probability measure of compact support on the set $\mathbb{P}_n$ of all positive definite matrices, let $t\in(0,1]$, and let $P_t(\mu)$ be the unique positive solution of $X=\int_{\mathbb{P}_n}X\sharp_t Z d\mu(Z)$. In this paper, we  show that
 $$ P_t(\mu)\leq I\quad \Longrightarrow\quad P_{\frac{t}{p}}(\nu)\leq P_t(\mu)$$
for every $p\geq1$, where $\nu(Z)=\mu(Z^{1/p})$. This provides an extension of the Ando--Hiai inequality for matrix power means. Moreover, we prove that if $\Phi:\mathbb{M}_n\to\mathbb{M}_m$ is a unital positive linear map, then $\Phi(P_t(\mu))\leq P_t(\nu)$ for all $t\in[-1,1]\backslash\{0\}$, where $\nu$ is a certain measure.
\end{abstract}
\maketitle

\section{Introduction and preliminaries}
Let $\mathbb{M}_n$ be the algebra of all $n\times n$ complex matrices and let $I$ be the identity matrix. We denote by $\mathbb{P}_n$ the set of all positive definite matrices in $\mathbb{M}_n$.  A real-valued continuous function $f$ defined on an interval $J\subseteq \mathbb{R}$ is said to be \emph{matrix convex} (\emph{matrix monotone}, resp.) if for all Hermitian matrices $A$ and $B$ with spectra in $J$, $f(\lambda A+(1-\lambda) B)\leq \lambda f(A)+(1-\lambda)f(B)$ for all $\lambda \in [0,1]$ ($A\leq B$ implies that $f(A)\leq f(B)$, resp.). Here $f(A)$ is defined by the continuous functional calculus.

The \emph{weighted geometric mean} of two positive definite matrices $A$ and $B$ is defined by $A\sharp_t B=B^{1/2}\left(B^{-1/2}AB^{-1/2}\right)^tB^{1/2}$, where $t\in[0,1]$. There have been some works devoted to introducing the matrix geometric mean of several variables; see \cite{ALM,bh,Ha}. Another multivariate matrix mean is the matrix power mean. A \emph{family of matrix power means} is introduced in \cite{LP} for every $k$-tuple of positive definite matrices $\mathbb{A}=(A_1,\ldots,A_k)$ and every $t\in(0,1]$ as the unique positive invertible solution of the matrix equation
\begin{align}\label{pme}
X=\sum_{i=1}^k \omega_i (X\ \sharp_{t}\ A_i),
\end{align}
which is denoted by $P_t(\omega;\mathbb{A})$. This family is defined also for $t\in [-1,0)$ by $P_t(\omega;\mathbb{A}) := P_{-t}(\omega;\mathbb{A}^{-1})^{-1}$, where $\mathbb{A}^{-1}=(A_1^{-1},\ldots,A_k^{-1})$.

A continuous form of matrix power means is studied in \cite{K-L} as follows: If $\mu$ is a probability measure of  compact support on $\mathbb{P}_n$ and $t\in{(0,1]}$, then
\begin{align}\label{eq}
X=\int_{\mathbb{P}_n}X\sharp_t Z d\mu(Z)
\end{align}
has a unique solution $P_t(\mu)$ in $\mathbb{P}_n$. This unique solution is called the \emph{power mean of $\mu$}. It defines a map $P_t$ from the set of all probability measures of compact support on $\mathbb{P}_n$ into $\mathbb{P}_n$. In the case when $t\in[-1,0)$, the power mean is defined by $P_t(\mu)=P_{-t}(\nu)^{-1}$, where $\nu(\mathcal{E})=\mu( \mathcal{E}^{-1})$ for every measurable set $\mathcal{E}$.
The integral above is in the sense of vector valued. If $f$ is a continuous function from a topological space $\mathcal{X}$ into a Banach space and $\mu$ is a probability measure of compact support on the Borel $\sigma$-algebra of $\mathcal{X}$, then the Bochner integral $\int_{\mathcal{X}} f d\mu$ is defined by the limit  $\lim_{m\to\infty}\sum_{i=1}^{N_m}f(a_{m,i})\mu(\mathcal{B}_{m,i})$ of the Riemannian sums in which $\{\mathcal{B}_{m,i}:\ i=1,\ldots,N_m\}$ is a partition of ${\rm supp}(\mu)$ and $a_{m,i}$ is an arbitrary point of $\mathcal{B}_{m,i}$ for each $1\leq i\leq N_m$.
Another approach to multivariate matrix means can be found in \cite{HSW}.

The well-known \emph{Ando--Hiai inequality} asserts that if $A\sharp_t B\leq I$ for two positive definite matrices $A$ and $B$, then $A^p\sharp_tB^p\leq I$ holds for every $p\geq1$. This interesting inequality has been investigated by several mathematician; see \cite{HSW, FIKM, LY, Wa}. Fujii and Kamei \cite{FK} proved that the Ando--Hiai inequality is equivalent to the Furuta inequality. Yamazaki \cite{YAM} extended it for the Riemannian mean of $n$ positive definite matrice. In addition, Seo \cite{SEO} presented a complementary to it. In this paper, we present an Ando--Hiai inequality for power means of probability measures. In particular, we show that if $\|P_t(\mu)\|\leq 1$, then $\|P_{\frac{t}{p}}(\nu)\|\leq 1$ for every $p\geq1$, where $\nu(Z)=\mu(Z^{1/p})$. Some known results are derived from our main result as special cases. Moreover, we prove that if $\Phi:\mathbb{M}_n\to\mathbb{M}_m$ is a unital positive linear map, then $\Phi(P_t(\mu))\leq P_t(\nu)$ for all $t\in[-1,1]\backslash\{0\}$, where $\nu$ is a certain measure.


\section{Main results}
We start our work with some lemmas which are needed to prove the main result.
 \begin{lemma}\label{llm1}
 Let $\mu$ be a probability measure of compact support on $\mathbb{P}_n$. Assume that $P_t(\mu)=X$. If $X\in {\rm supp}(\mu)$, then $P_t(\nu)=X$, where $\nu$ is defined on $ {\rm supp}(\mu)\backslash\{X\}$ by $\nu(\mathcal{E})=\frac{\mu(\mathcal{E}\backslash \{X\})}{1-\mu(\{X\})}$ for every Borel set $\mathcal{E}\subseteq\mathbb{P}_n$.
 \end{lemma}
 \begin{proof}
 Assume that $P_t(\mu)=X$ and $X\in {\rm supp}(\mu)$. Then $X=\int_{\mathbb{P}_n} X\sharp_t Zd \mu(Z)$. In addition, assume that $Y=P_t(\nu)$. We show that $Y=X$. It follows from the definition of power mean that $Y$ satisfies the equation $Y=\int_{\mathbb{P}_n} Y\sharp_t Zd \nu(Z)$. Then
 \begin{align*}
 Y=\lim_{n\to\infty} \sum_{i=1}^{N_n} Y\sharp_t Z_{n,i} \nu(\mathcal{B}_{n,i}),
 \end{align*}
 where $\{\mathcal{B}_{n,i}:\ i=1,\ldots, N_n\}$ is a partition of ${\rm supp}(\nu)$. We have
 \begin{align*}
 Y=\int_{\mathbb{P}_n} Y\sharp_t Zd \nu(Z) = \int_{{\rm supp}(\nu)} Y\sharp_t Z\frac{d \mu(Z)}{1-\mu(\{X\})}.
 \end{align*}
 It follows that
 \begin{align*}
 Y- \mu(\{X\})Y=\int_{{\rm supp}(\nu)} Y\sharp_t Z d \mu(Z)
 \end{align*}
 and so
 \begin{align*}
 Y&=\int_{{\rm supp}(\nu)} Y\sharp_t Z d \mu(Z)+ \mu(\{X\})Y=\int_{{\rm supp}(\nu)} Y\sharp_t Z d \mu(Z)+ \mu(\{X\})Y\sharp_t Y\\
 &=\lim_{n\to\infty} \sum_{i=1}^{N_n} Y\sharp_t Z_{n,i} \mu(\mathcal{B}_{n,i})+Y\sharp_t Y\mu(\{X\}).
 \end{align*}
 Since $X=P_t(\mu)$, we conclude that $X$ satisfies the above equation and so $X=Y$.
 \end{proof}


The proof of the next lemma is similar to that of Lemma \ref{llm1}, hence we omit it.

 \begin{lemma}\label{lm2}
 Let $\mu$ be a probability measure of compact support on $\mathbb{P}_n$. If $B\in {\rm supp}(\mu)$ and $P_t(\mu)\leq B$, then $P_t(\nu)\leq B$, where $\nu$ is defined on $ {\rm supp}(\mu)\backslash\{B\}$ by $\nu(\mathcal{E})=\frac{\mu(\mathcal{E}\backslash\{B\})}{1-\mu(\{B\})}$ for every Borel set $\mathcal{E} \subseteq \mathbb{P}_n$.
 \end{lemma}

\begin{lemma}\label{th1}
 Let $\mu$ be a probability measure of compact support on $\mathbb{P}_n$. If $t\in(0,1]$ and $\left(\int_{\mathbb{P}_n} Z^t d \mu(Z)\right)^\frac{1}{t}\leq I$, then $P_t(\mu)\leq I$.
 If $t\in[-1,0)$ and $\left(\int_{\mathbb{P}_n} Z^t d \mu(Z)\right)^\frac{1}{t}\geq I$, then $P_t(\mu)\geq I$.
\end{lemma}
\begin{proof}
Assume that $\left(\int_{\mathbb{P}_n} Z^t d \mu(Z)\right)^\frac{1}{t}\leq I$, which is equivalent to $\int_{\mathbb{P}_n} Z^t d \mu(Z)\leq I$. Then
$$\lim_{n\to \infty}\sum_{i=1}^{N_n}Z_{n,i}^t \mu(\mathcal{B}_{n,i}) \leq I,$$
where $\{\mathcal{B}_{n,i}:\ i=1,\ldots, N_n\}$ is a partition of ${\rm supp}(\mu)$.
It follows that there exists an expansive positive matrix $X$,  in the sense that $\|X\|\geq 1$, such that
\begin{align}\label{msm}
\lim_{n\to \infty}\sum_{i=1}^{N_n}Z_{n,i}^t\frac{\mu(\mathcal{B}_{n,i})}{2}+ \frac{1}{2}X^t= I.
\end{align}
Assume that $\nu$ is the measure on ${\rm supp}(\mu)\cup \{X\}$ defined by $\nu(\mathcal{E})=\mu(\mathcal{E})/2$ for every $\mathcal{E}\subseteq {\rm supp}(\mu)$ and $\nu (\{X\})=1/2$. Then $\{\mathcal{C}_{n,i}\}=\{\mathcal{B}_{n,i}:\ i=1,\ldots, N_n\}\cup\{X\} $ is a partition for ${\rm supp}(\nu)$, and we have
\begin{align*}
\int_{\mathbb{P}_n}I\sharp_t Z d \nu (Z)&= \lim_{n\to \infty}\sum_{i=1}^{N_n}I\sharp_t Z_{n,i} \nu(\mathcal{C}_{n,i})\\
&=\lim_{n\to \infty}\sum_{i=1}^{N_n}I\sharp_t Z_{n,i} \nu(\mathcal{B}_{n,i})+I\sharp_tX \nu(\{X\})\\
& =
\lim_{n\to \infty}\sum_{i=1}^{N_n}Z_{n,i}^t\nu(\mathcal{B}_{n,i})+ X^t \nu(\{X\})= I \qquad\qquad  \qquad\qquad  \qquad\qquad \qquad ({\rm by~} \eqref{msm}).
\end{align*}
The uniqueness of the solution of \eqref{eq} implies that $P_t(\nu)=I$. Now assume that $\lambda$ is defined on ${\rm supp}(\mu)\cup \{I\}$ by $\lambda(\mathcal{E})=\nu(\mathcal{E})$ for every $\mathcal{E}\subseteq {\rm supp}(\mu)$, and $\lambda(\{I\})=\frac{1}{2}$. Then $\{\mathcal{D}_{n,i}\}=\{\mathcal{B}_{n,i}:\ i=1,\dots, N_n\}\cup\{I\} $ is a partition for $\mathrm{supp(\lambda)}$. Note that
\begin{align*}
 \int_{\mathbb{P}_n} Y\sharp_t Z d\nu(Z)&=\lim_{n\to\infty} \sum_{i=1}^{N_n}Y\sharp_t Z_{n,i} \nu(\mathcal{C}_{n,i})\\
 &=\lim_{n\to\infty} \sum_{i=1}^{N_n}Y\sharp_t Z_{n,i} \mu(\mathcal{B}_{n,i})+\frac{1}{2}Y\sharp_t X\\
 &\geq \lim_{n\to\infty} \sum_{i=1}^{N_n}Y\sharp_t Z_{n,i} \mu(\mathcal{B}_{n,i})+\frac{1}{2}Y\sharp_t I\qquad(\mbox{by $X\geq I$})\\
 &=\lim_{n\to\infty} \sum_{i=1}^{N_n}Y\sharp_t Z_{n,i} \lambda(\mathcal{D}_{n,i})=\int_{\mathbb{P}_n} Y\sharp_t Z d \lambda(Z).
\end{align*}
 Assume that $f(Y)=\int_{\mathbb{P}_n} Y\sharp_t Z d\nu(Z)$ and $g(Y)=\int_{\mathbb{P}_n} Y\sharp_t Z d \lambda(Z)$. Then $f$ and $g$ are monotone functions and have unique fixed points $P_t(\nu)$ and $P_t(\lambda)$, respectively. Moreover, $f(Y)\geq g(Y)$ by the above equation. This implies that $f^k(Y)\geq g^k(Y)$ for every positive integer $k$ and so $P_t(\nu)\geq P_t(\lambda)$, that is, $P_t(\lambda)\leq I$. It follows from Lemma \ref{lm2} that $P_t(\mu)\leq I$.

Now assume that $t\in[-1,0)$ and $\left(\int_{\mathbb{P}_n } Z^t d \mu(Z)\right)^\frac{1}{t}\geq I$. Therefore, $\int_{\mathbb{P}_n} Z^t d \mu(Z)\leq I$. If the measure $\nu$ is defined by $d \nu(Z)=d\mu (Z^{-1})$, then
$$\int_{\mathbb{P}_n} Z^{-t} d \nu(Z)=\int_{\mathbb{P}_n} Z^{-t} d \mu(Z^{-1})=\int_{\mathbb{P}_n} W^{t} d \mu(W)\leq I.$$
We conclude from the first part of theorem that $P_{-t}(\nu)\leq I$. Hence
$$P_t(\mu)=P_{-t}(\nu)^{-1}\geq I.$$
\end{proof}
The next theorem gives the Ando--Hiai inequality for power means of probability measures. Recall that the \emph{matrix Jensen inequality} states that if $f$ is a matrix convex function on an interval $J$ and $\Phi$ is a unital positive linear map, then $f(\Phi(A))\leq \Phi(f(A))$ for all Hermitian matrices $A$ with spectrum in $J$. In particular, $f(C^*AC)\leq C^* f(A)C$ for every $C\in\mathbb{M}_n$ with $C^*C=I$.

\begin{theorem}\label{th2}
 Let $\mu$ be a compactly supported probability measure on $\mathbb{P}_n$ and let $t\in(0,1]$. If $\|P_t(\mu)\|\leq 1$, then $\|P_{\frac{t}{p}}(\nu)\|\leq 1$ for every $p\geq1$, where $\nu(Z)=\mu(Z^{1/p})$. \\
In particular,
 $$ P_{\frac{t}{p}}(\nu)\leq P_t(\mu) \quad \mbox{and}\quad P_{\frac{-t}{p}}(\nu)\geq P_{-t}(\mu).$$
\end{theorem}
\begin{proof}
Suppose that $\|P_t(\mu)\|\leq 1$, or equivalently, $P_t(\mu)\leq I$. If $X_t=P_t(\mu)$, then
$$X_t=\int_{\mathbb{P}_n} X_t\sharp_t Z d \mu (Z)=\int_{\mathbb{P}_n} X_t^\frac{1}{2}\left(X_t^\frac{-1}{2}ZX_t^\frac{-1}{2}\right)^tX_t^\frac{1}{2} d \mu (Z),$$
and so
$$\int_{\mathbb{P}_n} \left(X_t^\frac{-1}{2}ZX_t^\frac{-1}{2}\right)^t d \mu (Z)=I.$$
Let $p\in[1,2]$. Then the function $x\mapsto x^p$ is matrix convex and $x\mapsto x^\frac{t}{p}$ is matrix monotone. Since $X_t^\frac{-1}{2}$ is an expansive matrix, it follows from the matrix Jensen inequality that
$$\left(X_t^\frac{-1}{2}ZX_t^\frac{-1}{2}\right)^t=\left(X_t^\frac{-1}{2}ZX_t^\frac{-1}{2}\right)^{p \frac{t}{p}}\geq \left(X_t^\frac{-1}{2}Z^pX_t^\frac{-1}{2}\right)^\frac{t}{p},$$
whence
$$I=\int_{\mathbb{P}_n} \left(X_t^\frac{-1}{2}ZX_t^\frac{-1}{2}\right)^t d \mu (Z)\geq \int_{\mathbb{P}_n}\left(X_t^\frac{-1}{2}Z^pX_t^\frac{-1}{2}\right)^\frac{t}{p} d \mu (Z)=\int_{\mathbb{P}_n}Z^\frac{t}{p} d \lambda (Z),$$
in which $\lambda(Z)=\mu\left(\left(X_t^\frac{1}{2}ZX_t^\frac{1}{2}\right)^\frac{1}{p}\right)$. From Lemma \ref{th1} we deduce that $P_{\frac{t}{p}}(\lambda)\leq I$. Let $P_{\frac{t}{p}}(\lambda)=Y$. Then
\begin{align*}
 Y&=\int_{\mathbb{P}_n} Y\sharp_{\frac{t}{p}} Z d \lambda(Z)\\
 &=\int_{\mathbb{P}_n} Y\sharp_{\frac{t}{p}} Z d \mu\left(\left(X_t^\frac{1}{2}ZX_t^\frac{1}{2}\right)^\frac{1}{p}\right)\\
 &=\int_{\mathbb{P}_n} Y\sharp_{\frac{t}{p}} \left(X_t^\frac{-1}{2}Z^pX_t^\frac{-1}{2}\right) d \mu\left(Z\right)\\
 &=\int_{\mathbb{P}_n} Y\sharp_{\frac{t}{p}} \left(X_t^\frac{-1}{2}ZX_t^\frac{-1}{2}\right) d \nu\left(Z\right),
\end{align*}
where $\nu(Z)=\mu(Z^{1/p})$. Hence
$$X_t^\frac{1}{2}YX_t^\frac{1}{2}=\int_{\mathbb{P}_n} \left(X_t^\frac{1}{2}YX_t^\frac{1}{2}\right)\sharp_{\frac{t}{p}} Z d \nu\left(Z\right).$$
Therefore
$P_\frac{t}{p}(\nu)=X_t^\frac{1}{2}YX_t^\frac{1}{2}\leq X_t=P_t(\mu)$ for every $p\in[1,2]$. Employing this inequality with $\frac{t}{p}$ instead of $t$, we obtain
$P_\frac{ \frac{t}{p}}{p}(\nu_2)\leq P_\frac{t}{p}(\nu)$,
where $\nu_2(Z)=\nu(Z^{1/p})=\mu(Z^{1/p^2})$. It yields that $P_\frac{t}{p^2}(\nu_2)\leq P_t(\mu)$. Thus the inequality $P_\frac{t}{p}(\nu)\leq P_t(\mu)$ holds for every $p\geq 1$. Moreover, suppose that $\bar{\lambda}(\mathcal{E})=\lambda(\mathcal{E}^{-1})$ for every measure $\lambda$ and every measurable set $\mathcal{E}$. It follows from the first part of the proof that
\begin{align*}
P_{\frac{-t}{p}}(\nu) & = P_{\frac{t}{p}}(\bar{\nu})^{-1}\geq P_{t}(\bar{\mu})^{-1}=P_{-t}(\mu).
\end{align*}
\end{proof}

\begin{remark}
 Theorem \ref{th2} can be proved by results on matrix power means $P_{t}(\omega;\mathbb{A})$ and a convergence argument. Consider the Thompson metric $d_T$ defined by $d_T(A,B)=\|\log (A^{-1/2}BA^{-1/2})\|_\infty$ for every $A,B\in\mathbb{P}_n$ which is a complete metric on $\mathbb{P}_n$. Here $\|\cdot\|_\infty$ is the spectral norm. Now assume that $p\geq 1$ and  $t\in(0,1]$ and that $\mu$ is a compactly supported probability measure on $\mathbb{P}_n$. If follows from the compactness of ${\rm supp}(\mu)$ that for every $\varepsilon>0$, there exists a finite Borel partition $\{\mathcal{E}_1,\dots,\mathcal{E}_k\}$ of ${\rm supp}(\mu)$ such that $d_T(A,B)<\varepsilon$ and $d_T(A^p,B^p)<\varepsilon$ for all $A,B\in\mathcal{E}_j$,\ $1\leq j\leq k$. In other words, $e^{-\varepsilon} B\leq A\leq e^{ \varepsilon} B$ and $e^{-\varepsilon} B^p\leq A^p\leq e^{ \varepsilon} B^p$ for all $A,B\in\mathcal{E}_j$,\ $1\leq j\leq k$. For every $j=1,\dots,k$, choose any $A_j\in\mathcal{E}_j$ and make a finitely supported probability measure $\mu_\varepsilon:=\sum_{i=1}^{k}\mu(\mathcal{E}_j)\delta_{A_j}$. For positive scalar $\alpha$, assume that $\alpha \cdot \mu_\varepsilon$ is the push-forward of $\mu_\varepsilon$, that is, $(\alpha \cdot\mu_\varepsilon)(\mathcal{E})=\mu_\varepsilon(\alpha^{-1}\mathcal{E})$. Then
 \begin{align}\label{ja}
 e^{-\varepsilon} \mu_\varepsilon\leq \mu \leq e^{\varepsilon} \mu_\varepsilon \qquad\mbox{and}\qquad e^{-\varepsilon} \mu_\varepsilon^p\leq \mu^p \leq e^{\varepsilon} \mu_\varepsilon^p,
 \end{align}
 since for all $A\in\mathcal{E}_j,\ \ 1\leq j\leq k,$ we have $e^{-\varepsilon} A_j\leq A\leq e^{ \varepsilon} A_j$ and $e^{-\varepsilon} A_j^p\leq A^p\leq e^{ \varepsilon} A_j^p$.
 Note that if $\mu_1$ and $\mu_2$ are probability measures, then $\mu_1\leq \mu_2$ means $\mu_1(\mathcal{U})\leq \mu_2(\mathcal{U})$ for every open upper set $\mathcal{U}\subseteq \mathbb{P}_n$. It follows from \eqref{ja} and the monotonicity of $P_t$ (see \cite[Theorem 4.4]{K-L}) that
 \begin{align}\label{pja}
 e^{-\varepsilon} P_t(\mu_\varepsilon)\leq P_t(\mu) \leq e^{\varepsilon} P_t(\mu_\varepsilon) \qquad\mbox{and}\qquad e^{-\varepsilon} P_{\frac{t}{p}}(\mu_\varepsilon^p)\leq P_{\frac{t}{p}}(\mu^p) \leq e^{\varepsilon} P_{\frac{t}{p}}(\mu_\varepsilon^p).
 \end{align}
Note that from the definition of $\mu_\varepsilon$, we have $P_t(\mu_\varepsilon)=P_t(\omega;A_1,\dots,A_k)$ and $P_{\frac{t}{p}}(\mu_\varepsilon^p)=P_{\frac{t}{p}}(\omega;A_1^p,\dots,A_k^p)$, where $\omega=(\mu(\mathcal{E}_1),\dots,\mu(\mathcal{E}_k))$.

Now consider a positive sequence $\{\varepsilon_m\}$ tending to $0$. By taking $\mu_m=\mu_{\varepsilon_m}$, we can choose a sequence of finitely supported probability measures $\mu_m$ such that $P_t(\mu_m)\to P_t(\mu)$ and $P_{\frac{t}{p}}(\mu_m^p)\to P_{\frac{t}{p}}(\mu^p)$.

Finally, assume that $P_t(\mu)\leq I$. Put $\alpha_m=\|P_t(\mu_m)\|_\infty$. Evidently, $\alpha_m$ is a convergent sequence and $P_t(\alpha_m^{-1}\cdot\mu_m)\leq I$. It follows from \cite[Corollary 3.2]{LY} that $P_{\frac{t}{p}}((\alpha_m^{-1}\cdot\mu_m)^p)\leq I$, or equivalently, $ P_{\frac{t}{p}}(\mu_m^p)\leq \alpha_m^pI$. Letting $m\to\infty$, we get $P_{\frac{t}{p}}(\mu^p)\leq I$.
\end{remark}

\begin{corollary}\cite[Theorem 4.5]{K-L}\label{c1}
Let $\mu$ be a probability measure of compact support on $\mathbb{P}_n$. If $0<t\leq s\leq 1$, then
$$\left(\int_{\mathbb{P}_n}Z^{-1}d\mu(Z) \right)^{-1}\leq P_{-s}(\mu)\leq P_{-t}(\mu)\leq P_{ t}(\mu)\leq P_{s}(\mu)\leq \int_{\mathbb{P}_n}Z d\mu(Z). $$
\end{corollary}
\begin{proof}
 Assume that $t,s\in(0,1]$ with $t\leq s$. Setting $X_s=P_{ s}(\mu)$, we observe that $X_s=\int_{\mathbb{P}_n} X_s\sharp_s Z d\mu(Z)$. First, suppose that $s\leq 2t$. Passing to the riemannian sums and employing the matrix Jensen inequality applied to the matrix convex function $x\mapsto x^{s/t}$ we get
\begin{align*}
 I = \int_{\mathbb{P}_n} \left(X_s^{-1/2} ZX_s^{-1/2} \right)^s d\mu(Z)\geq \left(\int_{\mathbb{P}_n}\left(X_s^{-1/2} ZX_s^{-1/2} \right)^td\mu(Z)\right)^{s/t},
 \end{align*}
 which is equivalent to $\int_{\mathbb{P}_n}\left(X_s^{-1/2} ZX_s^{-1/2} \right)^td\mu(Z)\leq I$. Assume that $\nu$ is the measure on $\mathbb{P}_n$ defined by $d\nu(Z)=d\mu(X_s^{1/2} ZX_s^{1/2})$. Then we have $\int_{\mathbb{P}_n}Z^td\nu(Z)\leq I$. Lemma \ref{th1} now implies that $P_t(\nu)\leq I$. If $Y_t=P_t(\nu)$, then
 \begin{align*}
 Y_t=\int_{\mathbb{P}_n} Y_t\sharp_t Z d\nu(Z)=\int_{\mathbb{P}_n} Y_t\sharp_t Z d\mu(X_s^{1/2} ZX_s^{1/2})=\int_{\mathbb{P}_n} Y_t\sharp_t (X_s^{-1/2} ZX_s^{-1/2}) d\mu(Z).
 \end{align*}
 Therefore
 $$X_s^{1/2}Y_tX_s^{1/2}=\int_{\mathbb{P}_n} (X_s^{1/2}Y_tX_s^{1/2})\sharp_t Z d\mu(Z),$$
 that is, $P_t(\mu)=X_s^{1/2}Y_tX_s^{1/2}\leq X_s=P_s(\mu)$ for every $s\in[t,2t]$.

 Now assume that $2t< s\leq 4t$. Then $t\leq t'=2t\leq s\leq 2t'$. The first part implies that
 $P_t(\mu)\leq P_{t'}(\mu)\leq P_s(\mu)$. By Continuing this process, we conclude that $P_t(\mu)\leq P_s(\mu)$ for every $t\leq s$. Finally, we have
 $P_{-t}(\mu)= P_t(\bar{\mu})^{-1}\geq P_s(\bar{\mu})^{-1}=P_{-s}(\mu)$.
\end{proof}

As a particular case of Corollary \ref{c1}, assume that $\mathbb{A}=(A_1,\dots,A_k)$ is a $k$-tuple of positive matrices and that $\omega=(\omega_1,\dots,\omega_k)$ is a weight vector. Consider the probability measure $\mu$ on the set $\{A_1,\ldots,A_k\}\subseteq\mathbb{P}_n$ defined by $\mu(\{A_i\})=\omega_i$ for every $i=1,\dots,k$. If $X_t=P_t(\mu)$, then
$$X_t=\int_{\mathbb{P}_n}X_t\sharp_t Z d\mu(Z)=\sum_{i=1}^{k}\omega_i X_t\sharp_t A_i=P_t(\omega;\mathbb{A}).$$
Therefore, Corollary \ref{c1} implies that
$$\left(\sum_{i=1}^{k}\omega_iA_i^{-1}\right)^{-1}\leq P_{-s}(\omega;\mathbb{A})\leq P_{-t}(\omega;\mathbb{A}) \leq P_{t}(\omega;\mathbb{A})\leq P_{ s}(\omega;\mathbb{A})\leq\sum_{i=1}^{k}\omega_iA_i; $$
see \cite[Corollary 3.4]{LY}.

Another favorable property of matrix means is the information monotonicity of them via any positive linear map. By a theorem of Ando \cite{An}, if $\Phi:\mathbb{M}_n\to\mathbb{M}_m$ is a unital positive linear map, then $\Phi (X_t\sharp_t Z)\leq \Phi (X_t)\sharp_t \Phi (Z)$ for all $A,B\in\mathbb{P}_n$ and every $t\in(0,1]$. In the case of matrix power mean $P_t(\omega;\mathbb{A})$, this inequality is proved as $\Phi(P_t(\omega;\mathbb{A}))\leq P_t(\omega;\Phi(\mathbb{A}))$ for $t\in(0,1]$ in \cite{LP} and for $t\in[-1,0)$ in \cite{FSe}. We present this property for power means of probability measures. In the next lemma, we use the notion of the Tsallis relative matrix entropy, which is defined for all $A,B\in\mathbb{P}_n$ and every $t\in(0,1]$ by $T_t(A|B)=\frac{1}{t}(A\sharp_t B - A)$. Kamei \cite{Kamei} showed that the matrix power mean $P_t(\omega;\mathbb{A})$ can be considered as the unique solution of the equation $0=\sum_{i=1}^{k}\omega_i T_t(X|A_i)$ instead of \eqref{pme}; see also \cite{FSe} in the case when $t\in[-1,0)$.

\begin{lemma}\label{lm1}
Let $\mu$ be a probability measure of compact support on $\mathbb{P}_n$. For every $t\in[-1,1]\backslash\{0\}$
\begin{align*}
 \int_{\mathbb{P}_n}T_t(X|Z) d\mu(Z)\geq 0\quad \Longrightarrow\quad X\leq P_t(\mu),
\end{align*}
and
\begin{align*}
 \int_{\mathbb{P}_n}T_t(X|Z) d\mu(Z)\leq 0\quad \Longrightarrow\quad X\geq P_t(\mu).
\end{align*}
\end{lemma}
\begin{proof}
First, assume that $t\in(0,1]$, and consider the function $f(Y)=\int_{\mathbb{P}_n} Y\sharp_t Z d\mu(Z)$. It is easy to see that $f$ is monotone and $f^k(Y)\to P_t(\mu)$ as $k\to\infty$. Since $t\in(0,1]$, we have
\begin{align*}
 \int_{\mathbb{P}_n}T_t(X|Z) d\mu(Z)\geq 0\quad \Longleftrightarrow\quad X\leq \int_{\mathbb{P}_n} X\sharp_tZ d\mu(Z)=f(X).
\end{align*}
By the monotonicity of $f$, we obtain $f^k(X)\geq X$ for every positive integer $k$, and so $P_t(\mu)\geq X$ as required. \\
If $t\in[-1,0)$, then the statement $\int_{\mathbb{P}_n}T_t(X|Z) d\mu(Z)\geq 0$ is equivalent to $X\geq \int_{\mathbb{P}_n} X\natural_{t}Z d\mu(Z)$. Moreover, we can write
\begin{align*}
 X\geq \int_{\mathbb{P}_n} X\natural_{t}Z d\mu(Z)= \int_{\mathbb{P}_n} X (X^{-1}\sharp_{-t}Z^{-1})X d\mu(Z),
\end{align*}
that is, $\int_{\mathbb{P}_n} (Y\sharp_{-t}Z^{-1}) d\mu(Z)\leq Y$, where $Y=X^{-1}$. Assume that the measure $\nu$ is defined by $d\nu(Z)=d\mu(Z^{-1})$. Then $\int_{\mathbb{P}_n} (Y\sharp_{-t}Z) d\nu(Z)\leq Y$. Consider the function $g(Y)=\int_{\mathbb{P}_n} Y\sharp_{-t} Z d\nu(Z)$. Then $g(Y)\leq Y$. We have form the monotonicity of $g$ that $g^k(Y)\leq Y$ and so $P_{-t}(\nu)\leq Y$ as $k\to\infty$. Now, $P_t(\mu)=P_{-t}(\nu)^{-1}\geq Y^{-1}=X$ as desired.

The second assertion can be proved similarly.
\end{proof}

 \begin{remark}
 It should be noted that Lemma \ref{lm1} in a special case implies \cite[Theorem 3.1]{LY}. To see this, let $t\in(0,1]$. Then
 $$\left(\int_{\mathbb{P}_n} Z^t d\mu(Z)\right)^\frac{1}{t}\leq I \quad \Leftrightarrow \quad \int_{\mathbb{P}_n} Z^t d\mu(Z) \leq I \quad \Leftrightarrow \quad \int_{\mathbb{P}_n} I\sharp_tZ d\mu(Z)\leq I.$$
 This is equivalent to $\int_{\mathbb{P}_n}T_t(I|Z) d\mu(Z)\leq 0$. Lemma \ref{lm1} now ensures that $P_t(\mu)\leq I$, that is,
 \begin{align}\label{qbl}
\left(\int_{\mathbb{P}_n} Z^t d\mu(Z)\right)^\frac{1}{t}\leq I \quad \Longrightarrow\quad P_t(\mu)\leq I.
 \end{align}
 If $\mathbb{A}=(A_1,\dots,A_k)$ is a $k$-tuple of positive definite matrices and $\omega=(\omega_1,\dots,\omega_k)$ is a weight vector, then $\mu=\sum_{i=1}^{k}\omega_i\delta_{A_i}$ is a finitely supported probability measure and \eqref{qbl} turns to be
 $$\left(\sum_{i=1}^{k}\omega_i A_i^t\right)^\frac{1}{t}\leq I\quad \Longrightarrow\quad P_t(\omega;\mathbb{A})\leq I,$$
 which is \cite[Theorem 3.1]{LY}.

 \end{remark}
\begin{theorem}
Let $\mu$ be a compactly supported probability measure on $\mathbb{P}_n$. If $\Phi:\mathbb{M}_n\to\mathbb{M}_m$ is a unital positive linear map, then
 \begin{align}
 \Phi(P_t(\mu))\leq P_t(\nu)
 \end{align}
 for all $t\in[-1,1]\backslash\{0\}$, where $\nu$ is the measure defined by $\nu(\Phi(\mathcal{E}))=\mu(\mathcal{E})$ for all measurable set $\mathcal{E}\subseteq {\rm supp}(\mu)$, and $0$ otherwise.
\end{theorem}
\begin{proof}
 First, note that the measure $\nu$ is a well-defined measure whose support is equal to $\{\Phi(Z):\ \ Z\in{\rm supp}(\mu)\}$. Hence ${\rm supp}(\nu)$ is compact since $\Phi$ is continuous. Moreover, $\nu$ is a probability measure on $\mathbb{P}_m$. Indeed
 $$\int_{\mathbb{P}_m}d\nu(W)=\int_{{\rm supp}(\mu)}d\nu(\Phi(Z))=\int_{{\rm supp}(\mu)}d\mu(Z)=I.$$ Now, assume that $t\in(0,1]$ and $X_t=P_t(\mu)$. Clearly, $X_t$ is the unique fixed point of the function $f(X)=\int_{\mathbb{P}_n} X\sharp_t Z d\mu(Z)$. Since $\Phi$ is continuous, we have
 \begin{align*}
 \Phi(f(X_t))=\Phi(P_t(\mu))=\Phi\left(\int_{\mathbb{P}_n} X_t\sharp_t Z d\mu(Z)\right)=\int_{\mathbb{P}_n} \Phi (X_t\sharp_t Z) d\mu(Z).
 \end{align*}
We have $\Phi (X_t\sharp_t Z)\leq \Phi (X_t)\sharp_t \Phi (Z)$ and so
 \begin{align}\label{b1}
 \Phi(f(X_t))\leq \int_{\mathbb{P}_n} \Phi (X_t)\sharp_t \Phi (Z) d\mu(Z).
 \end{align}
 If the function $g$ is defined on $\mathbb{P}_m$ by $g(Y)=\int_{\mathbb{P}_m} Y\sharp_t W d\nu(W)=\int_{\mathbb{P}_n} Y\sharp_t \Phi (Z) d\mu(Z)$, then $g$ is monotone and has the unique fixed point $P_t(\nu)$. Moreover, it follows from \eqref{b1} that $\Phi(f(X_t))\leq g(\Phi(X_t))$. By the monotonicity of $g$ we get
 $$\Phi(f^2(X_t))=\Phi(f(f(X_t)))\leq g(\Phi(f(X_t)))\leq g^2(\Phi(X_t)).$$
 By induction, it holds that $\Phi(f^k(X_t))\leq g^k(\Phi(X_t))$ for every positive integer $k$ and so $\Phi(P_t(\mu))\leq P_t(\nu)$.
 Next assume that $t\in[-1,0)$ and $X_t=P_t(\mu)=P_{-t}(\overline{\mu})^{-1}$ in which $\overline{\mu}$ is the measure defined by $\overline{\mu}(\mathcal{E})=\mu(\mathcal{E}^{-1})$. Then $X_t=\left(\int_{\mathbb{P}_n} (X_t\sharp_{-t}Z)^{-1} d\mu(Z) \right)^{-1}$, or equivalently,
\begin{align*}
X_t^{-1}=\int_{\mathbb{P}_n} X_t^{-1}\sharp_{-t} Z^{-1} d\mu(Z)\quad & \Longleftrightarrow\quad X_t=\int_{\mathbb{P}_n} X_t(X_t^{-1}\sharp_{-t} Z^{-1})X_t d\mu(Z)
\\
& \Longleftrightarrow\quad X_t=\int_{\mathbb{P}_n} X_t \natural_{t}Z d\mu(Z)
\end{align*}
in which $A\natural_t B$ for $t\in[-1,0)$ is the matrix $t$-quasi geometric mean of $A$ and $B$  with the same formula as the  matrix geometric mean.
 Therefore
 \begin{align*}
 \Phi(P_t(\mu))=\Phi(X_t)&=\Phi\left(\int_{\mathbb{P}_n} X_t \natural_{t}Z d\mu(Z)\right)\\
 &=\int_{\mathbb{P}_n} \Phi(X_t \natural_{t} Z) d\mu(Z)\\
 &\geq \int_{\mathbb{P}_n} \Phi(X_t) \natural_{t}\Phi(Z) d\mu(Z),
 \end{align*}
 where the above inequality follows from the fact that if $t\in[-1,0)$, then $\Phi(A\natural_t B)\geq \Phi(A)\natural_t\Phi(B)$. This implies that $\int_{\mathbb{P}_n} \Phi(X_t) \natural_{t}\Phi(Z) d\mu(Z)- \Phi(P_t(\mu))\leq 0$, and so
 $$\int_{\mathbb{P}_n} T_t(\Phi(X_t)|\Phi(Z)) d\mu(Z)\geq 0, $$
 since $t\in[-1,0)$. It follows from the definition of $\nu$ that
 \begin{align*}
 0&\leq\int_{\mathbb{P}_n} T_t(\Phi(X_t)|\Phi(Z)) d\mu(Z)\\
 &= \int_{\mathbb{P}_n} T_t(\Phi(X_t)|\Phi(Z)) d\nu(\Phi(Z))\\
 &= \int_{\mathbb{P}_n} T_t(\Phi(X_t)|W) d\nu(W)
 \end{align*}
 Lemma \ref{lm1} then implies that $\Phi(X_t)\leq P_t(\nu)$.
\end{proof}


\end{document}